\documentclass[twoside]{article}
\usepackage{tikz,mdwlist, comment,amsfonts,amssymb,amsmath,amsthm,enumerate,hyperref}
\input{amssym.def}
\input{amssym}

\newcommand{\alert}[1]{\textbf{\color{red}
[[[#1]]]}\marginpar{\textbf{\color{red}**}}\typeout{ALERT:
\the\inputlineno: #1}}

\usetikzlibrary{calc}
\newtheorem{theorem}{Theorem}
\newtheorem{definition}{Definition}

\newtheorem{lemma}[theorem]{Lemma}
\newtheorem{proposition}[theorem]{Proposition}

\date{}
\title{The minimal degree of permutation representations of finite groups\\Amirim Honors Program final project}
\author{Oren Becker\\Advisor: Prof. Alexander Lubotzky\\The Hebrew University of Jerusalem, Israel}
\begin{document}
\maketitle

\section{Introduction}
In this thesis we study the following property, $\mu(G)$, of a finite group $G$:
\begin{definition}
$\mu(G) = min\{n \mid G \text{ embeds in } S_n\}$.
\end{definition}
By Cayley's theorem, $\mu(G) \leq |G|$.
We start, after the introduction, with an explicit formula for $\mu(G)$ when $G$ is abelian.
This formula and its proof first appeared in \cite{johnson71}. We give a different proof.
The formula shows that for abelian groups $G$ and $H$,
\begin{equation} \label{equation_alex_direct_product_equality}
\mu(G \times H) = \mu(G) + \mu(H)
\end{equation}

The equality \eqref{equation_alex_direct_product_equality} was established in \cite{wright75} for nilpotent groups
(and even more: for groups $G$ which contain a nilpotent subgroup $G_0$ such that $\mu(G)=\mu(G_0)$).
We extend it in Section \ref{section_additivity_for_central_socle}
to the class $CS$ of groups for which the socle is central
(and even more: for groups $G$ which contain a subgroup $G_0$ which belongs to $CS$ such that $\mu(G)=\mu(G_0)$).
We also study when $\mu(G)=|G|$ and begin to explore the compression ratio $cr(G)=\frac{|G|}{\mu(G)}$.
In \cite{johnson71}, it was determined when $cr(G)=1$. We refine it by showing that if $cr(G) > 1$ then $cr(G) \geq 1.2$
(this bound is tight).

\section{Background: Permutation Representations} \label{section_background}
Given a finite group G. A homomorphism $\rho:G \rightarrow S_n$ is called a permutation representation of G. In case $\rho$ is a monomorphism, we say $\rho$ is a faithful representation. The number n is called the degree of the representation $\rho$.
Any subgroup H $\leq$ G induces a transitive permutation representation of G by the action of G on the left cosets of H. That is, it induces a representation $\rho:G \rightarrow S_{Sym(G/H)}$, defined by $\rho(g) = (xH \mapsto gxH)$ for any $g \in G$. The degree of $\rho$ is $|G/H|=[G:H]$. The representation $\rho$ is faithful if and only if $Core_G(H) = 1$.
More generally, any multiset $\{H_1,...,H_m\}$ of subgroups of G induces a representation $\rho:G \rightarrow S_{Sym(G/H_1)} \times \dotsm \times S_{Sym(G/H_m)} \hookrightarrow S_{[G:H_1]+\dotsm+[G:H_m]}$ defined by $\rho(g) = ((x_1H_1, ..., x_mH_m) \mapsto (gx_1H_1, ..., gx_mH_m))$ for any $g \in G$. The representation $\rho$ is faithful if and only if $core_G(\cap_{i=1}^m H_i) = \cap_{i=1}^m core_G(H_i) = 1$.
The degree of $\rho$ is $\sum_{i=1}^m [G:H_i]$ and $\rho$ has $m$ transitive consituents.
Moreover, any permutation reprsentation of G is equivalent to a permutation representation induced by some multiset of subgroups in the way described above: Given a permutation representation $\rho$, an equivalent representation is induced by $\{H_1,...,H_m\}$, where $H_i$ is the point stabilizer of $\alpha_i$ and $\{\alpha_1,\dotsc,\alpha_m\}$ are representatives of the transitive consituents of $\rho$.
This correspondence between permutation representations and multisets of subgroups allows us to refer to such multisets
as a permutation representation and vice versa. We will use both viewpoints interchangeably.
A more detailed description of these basic results can be found in \cite{aschbacher86} (Chapter 2, p. 13).

\section{The Basics} \label{section_basics}
Given a representation $R=\{H_1,\dotsc,H_m\}$ of a finite group $G$, we denote by
$\mu_G(R)$ the degree of $R$ as a representation of $G$.
By the discussion in Section \ref{section_background} we have
$\mu_G(R) = \sum_{i=1}^n[G:H_i]$.
Thus we have a formula for the function $\mu$ given by
$\mu(G) = \min\{\sum_{H \in R}[G:H] \mid R \text{ is a collection of subgroups of $G$ with
$\cap_{H \in R} core_G(H) = 1$}\}$.

For any two nontrivial finite groups $G$ and $H$, we have $\mu(G \times H) \leq \mu(G) + \mu(H)$ because
for any pair of faithful representations, $R_1=\{G_1,\dotsc,G_n\}$ and $R_2=\{H_1,\dotsc,H_m\}$, of $G$ and $H$ respectively, we can construct the faithful representation $R=\{G_1 \times H,\dotsc,G_n \times H, G \times H_1, \dotsc, G \times G_m\}$ of $G \times H$, and $\mu_{G \times H}(R)=\mu_G(R_1)+\mu_H(R_2)$.

We proceed to explore some interaction between a representation of a group and representations of its subgroups
and then more specifically - between a representation of a direct product and representations of each of its factors.
One natural way to get a representation of $H$
is to restrict the representation of $G$ to the elements of $H$.
We now define a different way to induce a representation on a subgroup that will be useful for our purposes.
\begin{definition}[induced representation] \label{definition_induced_representation}
Let $G$ be a finite group, let $R=\{G_1,\dotsc,G_n\}$ be a representation of $G$ and let $H \leq G$ be a subgroup of $G$.
Then the representation $R_H=\{G_1 \cap H,\dotsc,G_n \cap H\}$ of H is called the \emph{induced representation by $R$ on $H$}.
\end{definition}

Warning: Even if $R$ is a faithful representation of $G$, $R_H$ is not necessarily a faithful representation of $H$.
For example, consider $G=S_3$, $R=\{\langle (1\text{ }2) \rangle\}$ and $H = \langle(1\text{ }2)\rangle$.


\begin{definition}[faithful decomposition] \label{definition_partitions_faithfully}
Let $G_1,\dotsc,G_n$ be finite groups and let $R$ be a faithful representation of $\prod_{i=1}^n G_i$.
We say that $(G_1,\dotsc,G_n)_R$ admits a \emph{faithful decomposition} as $R=\uplus_{i=1}^n R^{(i)}$
if for each $1 \leq i \leq n$, $R^{(i)}=\{\pi_i^{-1}(G^{(i)}_1),\dotsc,\pi_i^{-1}(G^{(i)}_{n_i})\}$
for some faithful representation $\{G^{(i)}_1,\dotsc,G^{(i)}_{n_i})\}$ of $G_i$.
\end{definition}

\begin{definition}[weak faithful decomposition] \label{definition_weakly_partitions_faithfully}
Let $G_1,\dotsc,G_n$ be finite groups and let $R$ be a faithful representation of $\prod_{i=1}^n G_i$.
We say that $(G_1,\dotsc,G_n)_R$ admits a \emph{weak faithful decomposition} as $R=\uplus_{i=1}^n R^{(i)}$
if for each $1 \leq i \leq n$, the induced representation $R^{(i)}_{G_i}$ is a \emph{faithful} representation of $G_i$.
\end{definition}

It is easy to see that if $(G_1,\dotsc,G_n)_R$ admits a faithful decomposition then it also admits
a weak faithful decomposition as the names imply.
If $G$ and $H$ are nontrivial finite groups and $(G,H)_R$ admits a
faithful decomposition as $R=R' \uplus R''$, we immediately conclude that
\begin{equation} \label{eq_direct_product_faithful_decomposition}\mu_{G \times H}(R) =
\mu_G(R_G') + \mu_H(R_H'') \end{equation}.
We now show that even if we only require $(G,H)_R$ to admit a \emph{weak} faithful decomposition,
we still get one inequality between the two sides of equality \eqref{eq_direct_product_faithful_decomposition}.

\begin{lemma}[weak decomposition inequality] \label{lemma_weak_partition_inequality}
Let $G$ and $H$ be nontrivial finite groups such that $(G, H)_R$ admits a weak faithful decomposition
as $R=R' \uplus R''$.
Then $\mu_{G \times H}(R) \geq \mu_G(R_G') + \mu_H(R_H'')$.
\end{lemma}
\begin{proof}
For each $K \in R'$, we have $[G:K \cap G] = \frac{|G|}{|K \cap G|} \leq
\frac{|G|}{|K \cap G|} \frac{|G||H|}{|KG|} =
\frac{|G|}{|K \cap G|} \frac{|H| \frac{|K||G|}{|KG|}}{|K|} =
\frac{|G|}{|K \cap G|} \frac{|H||K \cap G|}{|K|} = \frac{|G||H|}{|K|} = [G \times H:K]$.
Similiarly, for each $K \in R''$, we have $[H:K \cap H] \leq [G \times H:K]$.
Finally,
$\mu_{G \times H}(R) = \sum_{K \in R}[G \times H:K] \geq
\sum_{K \in R'}[G:K \cap G] + 
\sum_{K \in R''}[H:K \cap H]=
\mu_{G}(R_G') + \mu_{H}(R_H'')$ as desired.
\end{proof}

\begin{lemma} \label{lemma_minimal_representation_weakly_partitions_implies_direct_product_formula}
Let $G$ and $H$ be finite groups.
Then if there is a minimal-degree faithful representation $R$ of $G \times H$ such that $(G, H)_R$
admits a weak faithful decomposition then $\mu(G \times H) = \mu(G) + \mu(H)$.
The other direction is easy and is discussed in the begnning of this section.
\end{lemma}
\begin{proof}
Let $(G, H)_R$ admit a weak faithful decomposition as $R=R' \uplus R''$. We have
$\mu(G \times H) = \mu_{G \times H}(R) \overset{(a)}{\geq}
\mu_G(R_G') + \mu_H(R_H'') \geq \mu(G) + \mu(H)$
where inequality $(a)$ is due to Lemma \ref{lemma_weak_partition_inequality}.
\end{proof}

\begin{lemma} \label{lemma_minimal_representation_coprime_orders_partitions}
Let $G$ and $H$ be nontrivial finite groups such that $gcd(|G|, |H|)=1$
and let $R=\{K_1,\dotsc,K_n\}$ be a minimal-degree faithful representation of $G \times H$.
Then $(G, H)_R$ admits a faithful decomposition.
\end{lemma}
\begin{proof}
Since $gcd(|G|, |H|)=1$, by Lemma \ref{lemma_subgroups_of_direct_product_with_coprime_orders} in the appendix,  for any $1 \leq i \leq n$ we have $K_i = G_i \times H_i$ for some $G_i \leq G$, $H_i \leq H$.
Let $1 \leq i_0 \leq n$.
We have $K_{i_0} = G_{i_0} \times H_{i_0} = (G \times H_{i_0}) \cap  (G_{i_0} \times H)$.
Thus $R'=(R\setminus\{K_{i_0}\})\cup\{(G \times H_{i_0}), (G_{i_0} \times H)\}$ is a faithful representation of G (because R is).
Since R is of minimal degree, we have $0 \leq \mu_G(R')-\mu_G(R)=-[G \times H:G_{i_0} \times H_{i_0}] + [G \times H:G_{i_0} \times H] + [G \times H:G \times H_{i_0}] = [G:G_{i_0}][H:H_{i_0}]-[G:G_{i_0}]-[H:H_{i_0}]$.
That is, $[G:G_{i_0}]+[H:H_{i_0}] \geq [G:G_{i_0}][H:H_{i_0}]$. Thus either $[G:G_{i_0}]=1$ or $[H:H_{i_0}]=1$ or $[G:G_{i_0}]=[H:H_{i_0}]=2$.
But the latter is impossible since $gcd(|G|, |H|)=1$.
Thus either $K_{i_0} = G \times H_{i_0}$ or $K_{i_0} = G_{i_0} \times H$.
So $R=R' \uplus R''$ where $R'=\{G_1 \times H,\dotsc,G_r \times H\}$
and $R''=\{G \times H_1,\dotsc,G \times H_l\}$. We have
$1=
\cap_{i=1}^n core_{G \times H}K_i=
(\cap_{i=1}^r core_{G \times H}(G_i \times H)) \cap (\cap_{i=1}^l core_{G \times H}(G \times H_i))=
(\cap_{i=1}^r core_G(G_i) \times H) \cap (\cap_{i=1}^l G \times core_H(H_i))=
(core_G(\cap_{i=1}^r G) \times H) \cap (G \times core_H(\cap_{i=1}^l H_i))$.
Therefore, $core_G(\cap_{i=1}^r G_i) = core_H(\cap_{i=1}^l H_i) = 1$, so $\{G_1,\dotsc,G_r\}$ and $\{H_1,\dotsc,H_l\}$
are faithful representations of $G$ and $H$ respectively and so $(G,H)_R$ admits a faithful decomposition as claimed.
\end{proof}

\begin{theorem}[coprime additivity of $\mu$] \label{theorem_mu_for_direct_product_coprime_orders}
Let $G$ and $H$ be finite groups such that $gcd(|G|, |H|)=1$, then
$\mu(G \times H) = \mu(G) + \mu(H)$.
\end{theorem}
\begin{proof}
By Lemma \ref{lemma_minimal_representation_coprime_orders_partitions} we have that any minimal-degree
faithful representation of $G \times H$ admits a faithful decomposition.
In particular, there is a faithful representation which admits a faithful decompositions and therefore
admits a weak faithful decomposition
and thus by Lemma \ref{lemma_minimal_representation_weakly_partitions_implies_direct_product_formula}
we have $\mu(G \times H) = \mu(G) + \mu(H)$.
\end{proof}

To conclude this section we prove another basic result that shows that any finite group has a minimal-degree
representation with a certain useful property.\\
Recall that in any lattice $L$, an element $x \in L$ is called {\it meet-irreducible} if for any two elements
$y,z \in L$, $x = y \wedge z$ implies $x=y$ or $x=z$.

The following result first appeared as Lemma 1 in \cite{johnson71}.
\begin{proposition}[existence of a minimal-degree representation by meet-irreducile subgroups] \label{proposition_minimal_representation_by_meet_irreducible_subgroups}
Let $G$ be a finite group. Then there is a minimal-degree faithful permutation representation of $G$, given by $\{G_1,\dotsc,G_n\}$,
such that for any $1 \leq i \leq n$, $G_i$ is meet-irreducible in the subgroup lattice of G.
\end{proposition}
\begin{proof}
Let $R=\{K_1,\dotsc,K_m\}$ be a minimal-degree faithful permutation representation of $G$. That is - $\mu_G(R) = \mu(G)$.
First we note that
\begin{equation} \label{eq_minimal_representation_antichain}
\text{for any }1 \leq i < j \leq m\text{, we have }K_i \not\subset K_j\text{ and  }K_j \not\subset K_i
\end{equation}
In particular R is a set (not a multiset).\\
We will iteratively alter $R$ until all of the subgroups in it are meet-irreducible.
On one hand we will prove that each iteration keeps $R$ faithful of minimal degree.
On the other hand we will show that this iterative process terminates after some finite number of steps.
Together these 2 claims prove the existence of a minimal-degree faithful representation with the desired property.

We now describe the iterative process. As long as there is a meet-reducible subgroup of $G$ in $R$ we do the following: Let $K \in R$ be such a meet-reducible group.
So there are subgroups $M$ and $L$ of $G$ such that $K$ is a proper subgroup of both $M$ and $L$, but $K = M \cap L$.
Therefore $R'=(R \setminus \{K\}) \cup \{M, L\}$ is a faithful representation of $G$ with $\mu_G(R') = \mu(G) - [G:K] + [G:M] + [G:L] = 
\mu(G) + [G:K] (-1 + \frac{1}{[M:K]} + \frac{1}{[L:K]}) \leq \mu(G) + [G:K](-1 + \frac{1}{2} + \frac{1}{2}) = \mu(G)$.
Thus $R'$ is still a minimum-degree faithful representation.

It remains to show that this process eventually terminates. By property \eqref{eq_minimal_representation_antichain}
we know that  it is not possible to get the same representation in 2 different iterations.
But G is finite, and thus so is its subgroup lattice and therefore so is the number of subsets of its subgroup lattice and therefore the process does eventually terminate.
\end{proof}

It should be noted that the above proof shows that for a group $G$ of odd order, any minimal-degree
faithful representation is given by a collection of meet-irreducible subgroups.
We will not use that fact.

\section{The value of $\mu(G)$ for an abelian group $G$} \label{section_abelian_groups}
In this section we show how to compute the value of $\mu(G)$ for a finite abelian group $G$.
To describe the formula, we first need to recall that any finite abelian group is isomorphic to the direct
product of cyclic groups, each of prime-power order.
That is, if $G$ is a nontrivial finite abelian group then $G \cong \prod_{i=1}^n \mathbb{Z}_{p_i^{e_i}}$
for some integer $n \geq 1$, primes $p_1,\dotsc,p_n$ and integers $e_1,\dotsc,e_n \geq 1$.
This decomposition of $G$ is called the \emph{primary decomposition of $G$}.
Further, the primary decomposition of $G$ is unique up to the order of the factors.
This allows us to give a formula for $\mu(G)$ in terms of the numbers
$n$, $p_1,\dotsc,p_n$ and $e_1,\dotsc,e_n$.
We can now state the result of this section:
For an arbitrary finite abelian group $G$, isomorphic to $\prod_{i=1}^n \mathbb{Z}_{p_i^{e_i}}$, as above,
we have $\mu(G) = \sum_{i=1}^n p_i^{e_i}$.
This is the content of Theorem \ref{theorem_mu_for_abelian_groups}.
This result was first proved Theorem 2 of \cite{johnson71} by induction on the number of factors in the primary
decomposition of $G$.
We give a new, different, proof.

We begin with some notation:
\begin{definition}[the function $m$] \label{definition_m_function}
Let G be a finite abelian group. Let the unique primary decomposition of G be $G \cong \prod_{i=1}^n \mathbb{Z}_{p_i^{e_i}}$ for some $n \geq 1$, primes $p_1,\dotsc,p_n$ and integers $e_1,\dotsc,e_n \geq 1$.
Then we define m(G) := $\sum_{i=1}^n p_i^{e_i}$
\end{definition}
\begin{lemma}[properties of $m$] \label{lemma_m_function}
Let $K$ and $L$ be finite abelian groups and let $K=\prod_{i=1}^n \mathbb{Z}_{p_i^{d_i}}$, $L=\prod_{i=1}^m \mathbb{Z}_{q_i^{e_i}}$ be their primary decompositions, then:
\begin{itemize}
	\item (cardinality bound) $m(K) \leq |K|$.
	\item (additivity) $m(K \times L) = m(K) + m(L)$.
	\item (monotonicity) If $H \leq K$ then $m(H) \leq m(K)$.
\end{itemize}
\end{lemma}
\begin{proof}
\begin{itemize}
	\item $m(K) = \sum_{i=1}^n p_i^{d_i} \leq \prod_{i=1}^n p_i^{d_i} \leq |K| $.
	\item $m(K \times L) =  \sum_{i=1}^n p_i^{d_i} +  \sum_{i=1}^m q_i^{e_i} = m(K) + m(L)$
	\item
Lemma \ref{lemma_subgroups_of_abelian_group} in the appendix
states that if $H$ is a subgroup of the finite abelian group $K=\prod_{i=1}^n \mathbb{Z}_{p_i^{d_i}}$ then
$H=\prod_{i=1}^n \mathbb{Z}_{p_i^{a_i}}$ for some integers $a_1,\dotsc,a_n$
such that $0 \leq a_i \leq d_i$ for each $1 \leq i \leq n$.
Thus, $m(H) = m(\prod_{i=1}^n \mathbb{Z}_{p_i^{a_i}}) =
\sum_{\substack{1 \leq i \leq n\\a_i \neq 0}} p_i^{a_i} \leq
\sum_{1 \leq i \leq n} p_i^{a_i} \leq
\sum_{1 \leq i \leq n} p_i^{d_i} = m(K)$.
\end{itemize}
\end{proof}

\begin{lemma}[minimal degree of an abelian $p$-group]\label{lemma_mu_of_abelian_p_groups}
Let G be a finite abelian p-group. Then $\mu(G) = m(G)$.
\end{lemma}
\begin{proof}
Let the primary decomposition of G be $G\cong\prod_{i=1}^n \mathbb{Z}_{p_i^{d_i}}$.
We prove both {$\mu(G) \leq m(G)$} and {$\mu(G) \geq m(G)$} to conclude the desired equality:
\begin{itemize}
\item
{$\mu(G) \leq m(G)$}:
We need to construct a faithful permutation representation of $G$ of degree $m(G)$.
For any $1 \leq j \leq n$ define $H_i = \prod_{i=1}^{j-1}\mathbb{Z}_{p_i^{d_i}} \times 1 \times \prod_{i=j+1}^n\mathbb{Z}_{p_i^{d_i}}$.
Then the representation $R=\{H_1,\dotsc,H_n\}$ of $G$ is faithful because $\cap_{i=1}^{n}K_i = 1$ and its degree is
$d_G(R) = \sum_{i=1}^{n}[G:H_i] = \sum_{i=1}^n p_i^{d_i}=m(G)$, as required.

\item
{$\mu(G) \geq m(G)$}: We need to take an arbitrary faithful representation of $G$ and prove that its degree is no less than $m(G)$. Let $\{H_1,\dotsc,H_m\}$ be a faithful representation of $G$. It is sufficient to justify the following chain of equalities and inequalities:\\
$\mu_G(\{H_1,\dotsc,H_m\}) = \sum_{i=1}^m [G:H_i] = \sum_{i=1}^m |G/H_i| \overset{(a)}{\geq} \sum_{i=1}^m m(G/H_i) \overset{(b)}{=} m(\prod_{i=1}^m (G/H_i)) \overset{(c)}{\geq} m(G)$.\\
Steps $(a)$, $(b)$ and $(c)$ are due to the properties of the function $m$ stated in Lemma \ref{lemma_m_function}:
Inequality $(a)$ follows from the cardinality bound of $m$.
Equality $(b)$ follows from the additivity of $m$.
In order to show that inequality $(c)$ follows from the monotonicity of $m$ we need to show that
G embeds in $\prod_{i=1}^m (G/H_i)$ which we do as follows:\\
The function $\phi:G \rightarrow \prod_{i=1}^m (G/H_i)$ defined by $\phi(g) = (gH_1,\dotsc,gH_m)$ is a homomorphism. We have $ker(\phi) = \cap_{i=1}^m H_i \overset{(d)}{=} \cap_{i=1}^m core_G(H_i) \overset{(e)}{=} 1$.
Equality $(d)$ follows because $G$ is abelian and equality $(e)$ follows because the representation $\{H_1,\dotsc,H_m\}$ is faithful.  Thus $\phi$ is an embedding. Thus G embeds in $\prod_{i=1}^m (G/H_i)$ as desired. This completes the proof.
\end{itemize}
\end{proof}

\begin{theorem}[minimal degree of an abelian group] \label{theorem_mu_for_abelian_groups}
Let G be a finite abelian group. Then $\mu(G) = m(G)$.
\end{theorem}
\begin{proof}
The group $G$ is a direct product of abelian $p$-groups $G=\prod_{i=1}^n G_i$
where $G_i$ is a $p_i$-group for some distinct primes $p_1,\dotsc,p_n$.
We now have
$\mu(G) = \mu(\prod_{i=1}^n G_i) \overset{(1)}{=} \prod_{i=1}^n \mu(G_i) \overset{(2)}{=}
\prod_{i=1}^n m(G_i) \overset{(3)}{=} m(\prod_{i=1}^n G_i) = m(G)$,
where equality $(1)$ follows from the coprime addivity of $\mu$ proved in Lemma \ref{theorem_mu_for_direct_product_coprime_orders},
equality $(2)$ follows from the equality between $\mu$ and $m$ for abelian $p$-groups proved in Lemma \ref{lemma_mu_of_abelian_p_groups} above and
equality $(3)$ follows from the addivity of the function $m$ stated in the second part of Lemma \ref{lemma_m_function}.
\end{proof}

Note that if $G$ and $H$ are finite abelian groups, then by Theorem \ref{theorem_mu_for_abelian_groups},
we have $\mu(G \times H) = \mu(G) + \mu(H)$. A larger collection of groups for which this formula holds is the subject
of the next two sections.

\section{Additivity of $\mu$ for central socle groups} \label{section_additivity_for_central_socle}
This section generalizes a result first proved in \cite{wright75}.
Some of the ideas presented here are based on ideas which first appeared in \cite{wright75}.

Recall that the socle of a group $G$, denoted $Soc(G)$,
is the subgroup generated by all minimal normal subgroups of $G$.
The socle of a finite group is always a direct product of simple groups and thus,
if $G$ is a finite group and $Soc(G)$ is abelian, then $Soc(G)$ is the direct product of elementary abelian groups.

\begin{definition}[Central socle groups]
The collection $CS$ is defined as the collection of all nontrivial finite groups for which the socle is central.
That is, $CS := \{G \mid G \text{ is a nontrivial finite group and } Soc(G) \leq Z(G)\}$.
\end{definition}

For further discussion of the socle see Subsecion \ref{subsection_group_theory}.
In particular, in Lemma \ref{lemma_socle_of_direct_product} we show that $Soc(G \times H) = Soc(H) \times Soc(H)$ and
thus $CS$ is closed under taking direct products.

The purpose of this section is to prove the formula $\mu(G \times H) = \mu(G) + \mu(H)$
for any two groups $G$ and $H$ which belong to $CS$.
This is a generalization of the same formula for nilpotent groups given in \cite{wright75}
($CS$ strictly contains the collection of nilpotent groups).
It should be noted that there are pairs of groups $G$ and $H$ such that $\mu(G \times H) < \mu(G) + \mu(H)$.
For examples, see \cite{wright75} or \cite{saunders10}.
It should be also noted that in \cite{wright75}, after proving the formula for nilpotent groups, the same formula
is proved for an extended collection of groups, each containing a "large enough" nilpotent subgroup.
In the next section we employ the same extension mechanism,
thus proving the formula for groups which contain a "large enough" subgroup that belongs to $CS$.

We first outline the proof given in this section.
Consider two groups $G$ and $H$ which belong to $CS$.
By Lemma \ref{lemma_minimal_representation_weakly_partitions_implies_direct_product_formula},
it is enough to construct a minimal-degree faithful representation, $R$, of $G \times H$ such that $(G, H)_R$
admits a weak faithful decomposition.
We will show that if $R$ is a minimal-degree faithful representation of $G \times H$ given by a collection of
meet-irreducile subgroups of $G \times H$ then $(G, H)_R$ admits a weak-faithful decomposition.
As we have already seen in Lemma \ref{proposition_minimal_representation_by_meet_irreducible_subgroups},
such a representation exists.

We thus let $R$ be a minimal-degree faithful representation of $G \times H$ given
by meet-irreducible subgroups and proceed to show
that $(G, H)_R$ admits a weak faithful representation.
To do so, we consider the representation induced by $R$ on $Soc(G \times H)=Soc(G) \times Soc(H)$,
which we denote by $R_{Soc(G \times H)}$ as in Definition \ref{definition_induced_representation}.
We then prove that $R_{Soc(G \times H)}$ is faithful and that
$(Soc(G), Soc(H))_{R_{Soc(G \times H)}}$ admits a weak faithful decomposition as
$R_{Soc(G \times H)} = R_{Soc(G \times H)}' \uplus R_{Soc(G \times H)}''$.
Finally, we show that if $R'$ is the set of subgroups in $R$ which induce $R_{Soc(G \times H)}'$
and $R''$ is the set of subgroups in $R$ which induce $R_{Soc(G \times H)}''$ then $R=R' \uplus R''$
is a weak faithful decomposition of $(G, H)_R$, as desired.

Before executing the plan described above, we compare it to the proof given in \cite{wright75}.
In order to compare the two proofs we must describe the method of \cite{wright75} using the
terminology preseneted in Section \ref{section_basics}.
Both proofs start with a minimal-degree representation of $G \times H$ given by meet-irreducible subgroups.
In \cite{wright75} it is assumed that $G \times H$ is nilpotent and thus $R$ decomposes as faithful representations
of $p$-groups whose direct product is $G \times H$.
The proof in \cite{wright75} then proceeds in a method similar to the one used in our paper
to show that each of these representations decomposes to faithful representations
of a factor coming from $G$ and a factor coming from $H$,
using the fact that the socle of a $p$-group is a vector space.
Our paper refines this ideas by only requiring the socle to be central and immediately considering
the rerpesentation induced on $Soc(G \times H)$ (which is a direct product of vector spaces
when it is central). We then show that the representation induced on the
$Soc(G \times H)$ decomposes to faithful representations of $Soc(G)$ and $Soc(H)$ and show that
when we go back up to $G$ and $H$ we get faithful representations of $G$ and $H$.
To summarize the comparison, decomposing the representation entirely down at the socle, instead of first decomposing
to $p$-groups and then decomposing at the socle of each of them, is what allows us to generalize
the result proved in \cite{wright75}.

\begin{lemma}[properties of the induced representation on the socle] \label{lemma_induced_representation_of_central_socle_group_on_socle}
Let $G$ be a group belonging to $CS$ and let $R=\{G_1,\dotsc,G_n\}$ be a minimal-degree faithful representation of $G$.
Let $p_1,\dotsc,p_m$ be the set of primes dividing $|Z(G)|$.
Then the induced representation of $R$ on $Soc(G) = \prod_{i=1}^mZ(G)[p_m]$
has the following properties:
\begin{enumerate}
\item $R_{Soc(G)}$ is faithful and $(Z(G)[p_1],\dotsc,Z(G)[p_m])_{R_{Soc(G)}}$
decomposes faithfully as $R=\uplus_{i=1}^m R_i$ (denote $R_j = \{G_1^{(j)},\dotsc,G_{n_j}^{(j)}\}$).
\item For each $1 \leq j \leq m$, $R_j$ has no redundant transitive constituents. That is, for any $1 \leq i_0 \leq n_j$ we have $\cap_{i \neq i_0} (G_i^{(j)} \cap Z(G)[p_j]) \neq 1$
\suspend{enumerate}
Further, if $G_i$ is meet-irreducible in the subgroup lattice of $G$ for each $1 \leq i \leq n$, then:
\resume{enumerate}
\item For each $1 \leq j \leq m$ and $1 \leq i \leq n_j$, we have $dim(G_i^{(j)} \cap Z(G)[p_j]) = dim(Z(G)[p_j]) - 1$.
\end{enumerate}
\end{lemma}

\begin{proof}
\begin{enumerate}
\item
We first show that $R_{Soc(G)}$ is a faithful representation.
Assume, for the sake of contradiction, that $\cap_{i=1}^n (G_i \cap Soc(G)) \neq 1$.
We have $\cap_{i=1}^n (G_i \cap Soc(G)) \unlhd G$ because $\cap_{i=1}^n (G_i \cap Soc(G)) \leq Z(G)$.
But $\cap_{i=1}^n (G_i \cap Soc(G)) \leq \cap_{i=1}^n G_i$.
Thus $1 \neq \cap_{i=1}^n (G_i \cap Soc(G)) \leq core_G{\cap_{i=1}^n G_i}$
in contradiction with the faithfulness of $R=\{G_1,\dotsc,G_n\}$.
So $R_{Soc(G)}$ is faithful, and thus by Lemma
\ref{lemma_minimal_representation_coprime_orders_partitions}
it decomposes into faithful representations of $Z(G)[p_1],\dotsc,Z(G)[p_m]$ as desired.

\item Fix some $1 \leq j \leq m$ and $1 \leq i_o \leq n_j$ and assume, for the sake of contradiction,
that $\cap_{K \in {R_j\setminus\{G_i^{(j)}\}}} (K \cap Z(G)[p_j]) = 1$.
That is, $(\cap_{K \in {R_j\setminus\{G_i^{(j)}\}}} G_i^{(j)}) \cap Z(G)[p_j] = 1$.
Thus, by the faithful decomposition of $(Z(G)[p_1],\dotsc,Z(G)[p_m])_{R_{Soc(G)}}$ proved in conclusion $(1)$,
we get $(\cap_{K \in {R\setminus\{G_i^{(j)}\}}} G_i) \cap Soc(G) = 1$.
Then $core_G(\cap_{K \in {R\setminus\{G_i^{(j)}\}}} G_i) \cap Soc(G) = 1$.
Therefore $core_G(\cap_{K \in {R\setminus\{G_i^{(j)}\}}} G_i) = 1$.
Thus $R\setminus\{G_i^{(j)}\}$ is a faithful representation of G, contradicting the minimality of $R$.

\item Fix some $1 \leq j \leq m$ and $1 \leq i_o \leq n_j$ and assume, for the sake of contradiction,
that $dim(G_{i_0}^{(j)} \cap Z(G)[p_j]) = dim(Z(G)[p_j])$.
Then $G_{i_0}^{(j)} \cap Z(G)[p_j] = Z(G)[p_j]$. Thus, $\cap_{i \neq i_0} (G_i^{(j)} \cap Z(G)[p_j]) =
\cap_{i=1}^n (G_i^{(j)} \cap Z(G)[p_j]) = 1$, by conclusion $(1)$, contradicting conclusion $(2)$.\\
Assume, again - for the sake of contradiction, that $dim(G_{i_0}^{(j)} \cap Z(G)[p_j]) < dim(Z(G)[p_j]) - 1$.
Then $dim(Z(G)[p_j]/(G_{i_0}^{(j)} \cap Z(G)[p_j])) \geq 2$.
But $Z(G)[p_j]/(G_{i_0}^{(j)} \cap Z(G)[p_j]) \cong Z(G)[p_j]G_{i_0}^{(j)}/G_{i_0}^{(j)}$ by the second isomorphism theorem.
So $Z(G)[p_j]G_{i_0}^{(j)}/G_{i_0}^{(j)}$ is a vector space spanned by a basis $z_1G_{i_0}^{(j)},\dotsc,z_rG_{i_0}^{(j)}$ for some $r \geq 2$
and $z_1,\dotsc,z_r \in Z(G)$. In particular, $z_1G_{i_0}^{(j)}$ and $z_2G_{i_0}^{(j)}$ are
linearly independent and therefore
$span\{z_1G_{i_0}^{(j)}\} \cap span\{z_2G_{i_0}^{(j)}\} = \{G_{i_0}^{(j)}\}$ and $span\{z_1G_{i_0}^{(j)}\}, span\{z_1G_{i_0}^{(j)}\} \neq \{G_{i_0}^{(j)}\}$.
So $G_{i_0}^{(j)} < \langle z1,G_{i_0}^{(j)} \rangle, \langle z2,G_{i_0}^{(j)} \rangle$ and
$\langle z1,G_{i_0}^{(j)} \rangle \cap \langle z2,G_{i_0}^{(j)} \rangle = G_{i_0}^{(j)}$
contradicting the fact that $G_{i_0}^{(j)}$ is meet-irreducible.
\end{enumerate}
\end{proof}

\begin{lemma}[lifting a representation back up from $Soc(G)$ to $G$] \label{lemma_induced_representation_of_p_group_lifting}
Let $G$ be a group belonging to $CS$. Let $R=\{G_1,\dotsc,G_n\}$ be a representation of $G$.
Then if the induced representation $R_{Soc(G)}$ is a faithful representation of Soc(G),
then $R$ is a faithful representation of $G$. 
\end{lemma}
\begin{proof}
Assume, for the sake of contradiction, that $core_G \cap_{i=1}^n G_i \neq 1$.
Then $(core_G \cap_{i=1}^n G_i) \cap Soc(G) \neq 1$.
But then $1 \neq (core_G \cap_{i=1}^n G_i) \cap Soc(G) = (core_G \cap_{i=1}^n (G_i \cap Soc(G)) \leq
(core_{Soc(G)} \cap_{i=1}^n (G_i \cap Soc(G)))$ in contradiction to the faithfulness of $R_{Soc(G)}$.
\end{proof}

\begin{lemma} \label{lemma_flat_abelian_weakly_partition}
Let $r \geq 1$ be an integer and let $p_1,\dotsc,p_r$ be distinct primes.
Suppose that for each $1 \leq i \leq r$ we have:
\begin{itemize}
	\item $G_i$ and $H_i$ are elementary abelian $p_i$-groups of dimensions $m_i$ and $n_i$ respectively.
	\item $\{v_1^{(i)},\dotsc,v_{m_i+n_i}^{(i)}\}$ is a basis for the vector space $G_i \times H_i$
	\item $K_j^{(i)} = span\{v_k^{(i)} \mid 1 \leq k \leq m_i+n_i \text{ and } k \neq j\}$ for each $1 \leq j \leq m_i+n_i$.
\end{itemize}
Let $G = \prod_{i=1}^r G_i$ and $H = \prod_{i=1}^r H_i$ and
let $R = \{K_j^{(i)}\}_{\substack{1 \leq i \leq r\\1 \leq j \leq m_i+n_i}}$ be a (faithful) representation of $G \times H$.
Then $(G, H)_R$ admits a weak faithful decomposition.
\end{lemma}
\begin{proof}
By the formula for $\mu$ for abelian groups given in Lemma \ref{theorem_mu_for_abelian_groups}
we know that $R$ is a minimal-degree faithful representation of $G \times H$.
Therefore, since the orders of $G_i \times H_i$ and $G_j \times H_j$ are coprime whenever $i \neq j$
we conclude, by Lemma \ref{lemma_minimal_representation_coprime_orders_partitions},
that $(G_1 \times H_1,\dotsc,G_r \times H_r)_R$ admits a faithful decomposition.
It is thus sufficient to prove, for each $1 \leq i \leq r$, that $(G_i, H_i)_{R_{G_i \times H_i}}$
admits a weak faithful decomposition.
Fix some $1 \leq i_0 \leq r$ and denote
$m=m_{i_0}$, $n=n_{i_0}$ and $v_j = v_j^{(i_0)}$ for each $1 \leq j \leq m_{i_0}+n_{i_0}$.
Form a matrix $M$ that has $\{v_1,\dotsc,v_{m+n}\}$ as its rows.
By applying the matrix decomposition whose definition and existence are given in Lemma \ref{lemma_invertible_matrix_faithful_decomposition} in the appendix to the invertible matrix $M$ with parameter $m$, we can conclude that we can assume that the matrix M is of the form:
$
M=\left[
\begin{array}{cc}
A & B \\
C & D
\end{array}\right]
$
where $A$ is an $m \times m$ matrix, $D$ is an $(n-m) \times (n-m)$ matrix and both $A$ and $D$ are invertible.
Partition $R$ as $R = R' \biguplus R''$
where $R'=\{K_1,\dotsc,K_m\}$ and $R''=\{K_{m+1},\dotsc,K_{m+n}\}$.
Now $R'_G$ is a faithful representation of $G$ because
$\cap_{K \in R'_G}K' =
\cap_{K \in R'}(K \cap G) =
(\cap_{K \in R'}K) \cap G =
span\{v_{m+1},\dotsc,v_{m+n}\} \cap G\overset{(a)}{=} 1
$
where equality $(a)$ is due to the fact that the submatrix $D$ of $M$ is invertible.
Similarly, $R''_H$ is a faithful representation of $H$ and thus
$(G,H)_R$ admits a weak faithful decomposition as claimed.
\end{proof}

\begin{lemma} \label{lemma_mu_for_direct_product_of_central_socle_groups}
Let G and H be groups belonging to $CS$, then $\mu(G \times H) = \mu(G) + \mu(H)$.
\end{lemma}
\begin{proof}
Let $R=\{K_1,\dotsc,K_n\}$ be a minimal-degree faithful representation of $G \times H$.
By Lemma \ref{proposition_minimal_representation_by_meet_irreducible_subgroups} we can assume that
$K_1,\dotsc,K_n$ are all meet-irreducible.
The properties of the the induced representation on $Soc(G \times H)$ proved in
Lemma \ref{lemma_induced_representation_of_central_socle_group_on_socle}
together with the linear algebra result proved in Lemma \ref{lemma_vector_spaces_in_faithful_mode} in the appendix
show that $Soc(G)$, $Soc(H)$ and the representation $R_{Soc(G \times H)}$ fit the hypothesis
of Lemma \ref{lemma_flat_abelian_weakly_partition} and therefore
$(R_{Soc(G)}, R_{Soc(H)})_{R_{Soc(G \times H)}}$ admits a weak faithful decomposition.
Thus, by Lemma \ref{lemma_induced_representation_of_p_group_lifting},
$(G,H)_R$ admits a weak faithful decomposition too.
Therefore, by Lemma \ref{lemma_weak_partition_inequality} we conclude that
$\mu(G \times H) = \mu(G) + \mu(H)$ as claimed.
\end{proof}

\section{A larger collection for which $\mu$ is additive}
We now extend the collection $CS$ to a larger collection for which the function $\mu$ is additive.
The extended collection, denoted $CSE$, is defined as the collection of groups $G$ for which
there is a subgroup $H \leq G$ such that $H \in CS$ and $\mu(H) = \mu(G)$.
This extension idea first appeared in \cite{wright75}, where a collection $\mathfrak{G}$ was similarly defined as
the collection of groups $G$ for which there is a nilpotent subgroup such that $\mu(H) = \mu(G)$.
The collection $\mathfrak{G}$ is obviously a subcollection of $CSE$ since the collection of nilpotent groups
is a (proper) subcollection of $CS$. We show $\mathfrak{G}$ is a proper subcollection of $CSE$ by giving an example
of a group in $CSE$ (actually, in $CS$, which is subcollection of $CSE$) that does not belong to $\mathfrak{G}$.

We begin by proving that $CSE$ is closed under taking direct products and that
$\mu$ is additive for groups belonging to $CSE$.

\begin{lemma}
Let the groups $G$ and $H$ belong to the collection $CSE$.
Then:
\begin{enumerate}
	\item $G \times H$ belongs to $CSE$.
	\item $\mu(G \times H) = \mu(G) + \mu(H)$.
\end{enumerate}
\end{lemma}
\begin{proof}
On one hand $\mu(G \times H) \leq \mu(G) + \mu(H)$.
On the other hand, since $G$ and $H$ belong to $CSE$, there are subgroups $G_1$ and $H_1$
of $G$ and $H$ resepectively such that $G_1$ and $H_1$ both belong to $CS$ and
$\mu(G_1) = \mu(G)$ and $\mu(H_1) = \mu(H)$.
Therefore $\mu(G \times H) \overset{(a)}{\geq} \mu(G_1 \times H_1) = \mu(G_1) + \mu(H_1) = \mu(G) + \mu(H)$.
Thus $\mu(G \times H) = \mu(G) + \mu(H)$, proving conclusion $(2)$. Therefore inequality
$(a)$ is in fact an equality and thus $\mu(G_1 \times H_1) = \mu(G \times H)$ which proves conclusion $(1)$
because the subgroup $G_1 \times H_1$ of $G \times H$ belongs to $CS$
since $CS$ is closed under taking direct products.
\end{proof}

We proceed to show that the binary icosahedral group $SL(2,5)$ belongs to $CSE$, but not to $\mathfrak{G}$.
First, $SL(2,5)$ belongs to $CS$ (and thus to $CSE$) because its only proper normal subgroup is its center,
$\{+1,-1\}$.
To show that $SL(2,5)$ does not belong to $\mathfrak{G}$ we first note that its nilpotent subgroups are
isomorphic to cyclic groups of orders 1,2,3,4,5,6 and 10, or to $Q_8$.
Of these groups, the one with the largest minimal-degree is $Q_8$, for which $\mu(Q_8)=8$
(see Lemma \ref{theorem_incompressible_groups}).
Thus, it is sufficient to show that $\mu(SL(2,5)) > 8$.
In fact, we will show that $\mu(SL(2,5))=24$:
Recall that $Z(SL(2,5))=\{-1,+1\}$ is normal in $SL(2,5)$.
Thus any faithful representation of $SL(2,5)$ must be given by a collection of subgroups of $SL(2,5)$ of which at least
one does not contain the element $-1$.
But, the element $-1$ is the only element of order 2 in $SL(2,5)$.
Therefore, any subgroup of $SL(2,5)$ of even order contains the element $-1$.
Thus, any faithful representation of $SL(2,5)$ must be given by a collection of subgroups of which at least one is
of odd order. But the largest subgroup of $SL(2,5)$ of odd order is a cyclic group of order $5$.
So we can already conclude that the degree of any faithful representation of $SL(2,5)$ is at least 120/5=24.
Conversely, any subgroup of $SL(2,5)$ of odd order does not contain $Z(SL(2,5))$, which is the unique
minimal normal subgroup of $SL(2,5)$. Therefore the representation
$\{\mathbb{Z}_5\}$ is a minimal-degree representation of $SL(2,5)$ and its degree is 24.

\section{Semidirect Products} \label{section_semidirect_products}
\begin{lemma} \label{lemma_mu_semidirect_inequality}
Let $G$ and $H$ be nontrivial finite groups. Then $\mu(G \rtimes H) \leq |G| + \mu(H)$.
\end{lemma}
\begin{proof}
It is sufficent to embed $G \rtimes H$ in $Sym(G) \times H$.
Let the multiplication in $G \rtimes H$ be defined by $(g_1,h_1)(g_2,h_2) = (g_1\varphi_{h_1}(g_2), h_1h_2)$
for any $g_1,g_2 \in G$ and $h_1,h_2 \in H$ where $\varphi : H \rightarrow Aut(G)$ is a homomorphism.
We show that $\rho : G \rtimes H \rightarrow Sym(G) \times H$ defined by
$\rho(g_0, h_0) = ((g \mapsto g_0\varphi_{h_0}(g)), h_0)$ is a monomorphism.
The function $\rho$ is a homomorphism because
$\rho((g_1, h_1)(g_2, h_2)) = \rho((g_1\varphi_{h_1}(g_2), h_1h_2)) = (g \mapsto (g_1\varphi_{h_1}(g_2)\varphi_{h_1h_2}(g)), h_1h_2) =
(g \mapsto g_1\varphi_{h_1}(g_2\varphi_{h_2}(g)), h_1h_2) = (g \mapsto g_1\varphi_{h_1}(g), h_1)(g \mapsto g_2\varphi_{h_2}(g), h_2) =
\rho(g_1, h_1)\rho(g_2,h_2)$.
The homomorphism $\rho$ is injective because $\rho(g_0,h_0) = (id,1)$ implies $((g \mapsto g_0\varphi_{h_0}(g)),h_0) = (id,1)$, that is
$h_0 = 1$ and $(g \mapsto g_0g) = (g \mapsto g_0\varphi_{1}(g)) = (g \mapsto g_0\varphi_{h_0}(g)) = id$.
Thus we must have $g_0 = 1$ and so $ker(\rho) = (id, 1)$.
\end{proof}

\section{Compression Ratio} \label{section_compression_ratio}
\begin{definition}[compression ratio]
Let $G$ be a finite group.
Then the \emph{compression ratio} of $G$ is defined as $cr(G)=\frac{|G|}{\mu(G)}$
\end{definition}

For any finite group $G$ we have $cr(G) \geq 1$ since $\mu(G) \leq |G|$ by Cayley's theorem.

\begin{lemma}[monotonicity of compression ratio] \label{lemma_mu_by_subgroup}
Let $G$ be a finite group and let $H \leq G$ be a subgroup of $G$. Then $cr(H) \leq cr(G)$
\end{lemma}
\begin{proof}
The inequality $cr(H) \leq cr(G)$ is equivalent to the inequality $\mu(G) \leq [G:H]\mu(H)$.
Therefore, it is sufficient to construct a faithful permutation representation of $G$ of degree $[G:H]\mu(H)$.
Let $\{H_1,\dotsc,H_n\}$ be a minimal-degree permutation representation of $H$.
That is $core_H(\cap_{i=1}^{n}H_i) = 1$ and $\sum_{i=1}^n[H:H_i]=\mu_H(\{H_1,\dotsc,H_n\}) = \mu(H)$.
The representation $\{H_1,\dotsc,H_n\}$ can also be viewed as a representation of $G$.
We show that it is faithful and of the desired degree:
The faithfulness of $\{H_1,\dotsc,H_n\}$ as a representation of $G$ follows because $core_G(\cap_{i=1}^{n}H_i) \leq core_H(\cap_{i=1}^{n}H_i) = 1$ and thus $core_G(\cap_{i=1}^{n}H_i) = 1$.
The degree of $\{H_1,\dotsc,H_n\}$ as a representation of $G$ is $\mu_G(\{H_1,\dotsc,H_n\}) = \sum_{i=1}^n[G:H_i] = \sum_{i=1}^n[G:H][H:H_i] =
[G:H]\sum_{i=1}^n[H:H_i] = [G:H]\mu_H(\{H_1,\dotsc,H_n\}) = [G:H]\mu(H)$.
\end{proof}

A finite group $G$ is called $incompressible$ if $cr(G)=1$.
The following characterization of incompressible groups is due to \cite{johnson71}.
We strengthen the conclusion described in \cite{johnson71} by stating that if a group has a compression
ratio larger than 1, then its compression ratio is at least 1.2 (this is tight because $cr(\mathbb{Z}_6)=1.2$).

\begin{theorem}[incompressible groups] \label{theorem_incompressible_groups}
Let $G$ be a nontrivial finite group.
The following conditions are equivalent:
\begin{itemize}
	\item The group $G$ is incompressible (that is, $cr(G)=1$).
	\item The group $G$ is of one of the following types:
	\begin{enumerate}[(a)]
		\item Cyclic group of prime power order
		\item Generalized quaternion group of order $2^n$ (for $n \geq 3$)
		\item The Klein four-group $V_4$
	\end{enumerate}
\end{itemize}
Further, $cr(G) < 1.2$ if and only if $cr(G)=1$.
\end{theorem}
\begin{proof}
We begin by showing that groups of types $(a)$, $(b)$ or $(c)$ are incompressible.
If $G$ is of type $(c)$ then by the formula for the function $\mu$ for abelian groups
given in Theorem \ref{theorem_mu_for_abelian_groups} we have $\mu(G)=2+2=4=|G|$
and thus $G$ is incompressible.
Assume now that $G$ is either of type $(a)$ or of type $(b)$.
Then $G$ has a unique minimal subgroup $H$.
The subgroup $H$ must be normal in $G$.
Let $R=\{G_1,\dotsc,G_m\}$ be a faithful representation of $G$.
Assume, for the sake of contradiction, that for any $1 \leq i \leq m$ we have $G_i \neq 1$.
Then, for any $1 \leq i \leq m$, we have $H \leq G_i$.
Thus $H \leq \cap_{i=1}^m G_i$.
But $H$ is normal in $G$ and therefore $1 \neq H \leq core_G(\cap_{i=1}^m G_i)$
in contradiction to the faithfulness of the representation $R$.
Therefore, there is some $1 \leq i_0 \leq m$ such that $G_{i_0} = 1$.
Thus $\mu_G(R) = \sum_{i=1}^m [G:G_i] \geq [G:G_{i_0}] = [G:1] = |G|$.
Finally, since $R$ is an arbitrary faithful representation of $G$ we get $\mu(G)=|G|$ and
thus $G$ is incompressible.\\
To complete the proof we need to show that if $cr(G) < 1.2$ then $G$ is of one of the types
$(a)$, $(b)$ or $(c)$. Assume that $cr(G) < 1.2$.
We first show that if $H$ and $K$ are nontrivial subgroups of $G$ satisfying
$H \cap K = 1$, then both $H$ and $K$ are of order 2.
Assume for the sake of contradiction that $|H| \geq 3$.
The representation $R=\{H,K\}$ of $G$ is faithful because $H \cap K=1$.
Thus we have $\mu(G) \leq \mu_G(R) = [G:H]+[G:K] = |G|(1/|H|+1/|K|) \leq
|G|(1/3+1/2)=(5/6)|G|$. Therefore, $cr(G) \geq 1.2$, contradicting the assumption that $cr(G) < 1.2$.
Thus any two nontrivial subgroups of $G$ intersecting trivially must be both of order 2.
In particular, there cannot be two elements in $G$ of distinct prime orders and therefore
$G$ is a $p$-group for some prime $p$. If $p$ is an odd prime, then $G$ is a group of odd-order which
has a unique subgroup of order $p$ and thus $G$ is of type $(a)$ (see \cite{zassenhaus49}, p. 118. Theorem 15).
If $p=2$, that is, $G$ is a 2-group, we consider two cases:
If there is an element $g$ in $G$ of order 4 then $g^2$ must be the unique element of order 2 in $G$ and thus
we conclude that $G$ is either of type $(a)$ or of type $(b)$ (again, by \cite{zassenhaus49}, p. 118. Theorem 15).
If, on the other hand, no element of $G$ is of order 4 then $G$ in an elementary abelian 2-group.
That is $G = \mathbb{Z}_2^n$ for some $n \geq 1$ and thus, by the formula for the function $\mu$ for
abelian groups given in Theorem \ref{theorem_mu_for_abelian_groups} we have $\mu(G) = 2n$. But $|G|=2^n$.
Thus $1.2 > cr(G) = 2^n/(2n)$ and thus either $n=1$ or $n=2$. That is, $G$ is either of type $(a)$
or of type $(c)$.
\end{proof}

Note that if we further assume that $G$ is of odd order,
then, by similar reasoning, we get $cr(G) < 1.5$ if and only if $cr(G)=1$.

We believe it would be interesting to continue the study of the compression ratio by answering questions similar to the
following:

Is there a function $f:\mathbb{R} \rightarrow \mathbb{R}$
such that whenever $cr(G) \leq r$ there must be a solvable subgroup of $G$ of index $\leq f(r)$?

\section{Appendix}
\subsection{Group Theory} \label{subsection_group_theory}

\begin{lemma} \label{lemma_subgroups_of_direct_product_with_coprime_orders}
Let G and H be finite groups such that $gcd(|G|, |H|)=1$ and let $K \leq G \times H$. Then for some $G' \leq G$, $H' \leq H$ we have $K = G' \times H'$.
\end{lemma}
\begin{proof}
Let $G'=\pi_1(K)$, $H'=\pi_2(K)$ where $\pi_i$ is the projection of the $i^{th}$ coordinate. Obviously $K \subset G' \times H'$. For the reverse inclusion, let $(g,h) \in G' \times H'$.
Then there are $g' \in G$, $h' \in H$ such that $(g,h'), (g',h) \in K$. Since $gcd(|G|, |H|)=1$ and by the chinese remainder theorem,
there exist integers $e_1$, $e_2$ such that $e_1 \equiv 1 \pmod{|G|}$, $e_1 \equiv 0 \pmod{|H|}$, $e_2 \equiv 0 \pmod{|G|}$, $e_2 \equiv 1 \pmod{|H|}$.
Thus, $(g,1) = (g,h')^{e_1} \in K$ and $(1,h) = (g',h)^{e_2} \in K$ and so $(g,h) = (g,1)(1,h) \in K$.
\end{proof}

\begin{definition} \label{definition_G[m]}
If $G$ is abelian group and $m > 0$ is an integer, then $G[m] := \{x \in G \mid mx = 0\}$
\end{definition}

Note that if $p$ is prime then $G[p]$ is a vector space (over $\mathbb{Z}_p$).

\begin{definition} \label{definition_socle}
Let $G$ be a finite group. The the \emph{socle} of G, denoted $Soc(G)$,
is the subgroup of $G$ generated by the nontrivial minimal normal subgroups of $G$.
\end{definition}

\begin{lemma} \label{lemma_central_socle}
Let $G$ be a finite group for which $Soc(G) \leq Z(G)$.
Then $Soc(G) = \prod_{\substack{p \mid |Z(G)|\\\text{$p$ prime}}}Z(G)[p]$.
\end{lemma}
\begin{proof}
On one hand, any cyclic central subgroup of prime order of $G$ is a minimal normal subgroup.
On the other hand, since $Soc(G) \leq Z(G)$, any minimal normal subgroup of $G$ must be central
and thus must be cyclic of prime order.
So $Soc(G)$ is the subgroup generated by all central elements of $G$ of prime order which is
$\prod_{\substack{p \mid |Z(G)|\\\text{$p$ prime}}}Z(G)[p]$ as claimed.
\end{proof}

\begin{lemma} \label{lemma_socle_of_direct_product}
Let $G$ and $H$ be finite groups. 
Then $Soc(G \times H) = Soc(G) \times Soc(H)$.
\end{lemma}
\begin{proof}
On one hand, any minimal normal subgroup of $G$ or $H$ is a minimal normal subgroup of $G \times H$ and thus $Soc(G) \times Soc(H) \leq Soc(G \times H)$.
We proceed to show the reverse inclusion:
Let $1 \neq N \unlhd G \times H$ be a minimal normal subgroup of $G \times H$.
We need to show that $N \leq Soc(G) \times Soc(H)$.
Since $N \leq \pi_1(N) \times \pi_2(N)$, it is sufficient to prove that $\pi_1(N) \leq Soc(G)$ and $\pi_2(N) \leq Soc(H)$.
To do so, we will show that $\pi_1(N)$ and $\pi_2(N)$ are minimal normal subgroups of $Soc(G)$ and $Soc(H)$ respectively.
First, $\pi_1(N)$ is a normal subgroup of $G$ because for any $g \in \pi_1(N)$ there exists some $h \in H$ such that
$(g,h) \in N$ and therefore, for any $g_0 \in G$ it holds that $(g_0 g g_0^{-1},h) = (g_0,1)(g,h)(g_0,1)^{-1} \in N$
and therefore $g_0 g g_0^{-1} \in \pi_1(N)$. Similarly, $\pi_2(N)$ is a normal subgroup of $H$.
As for minimality, assume, for the sake of contradiction, that there exists $1 \neq N_1 < \pi_1(N)$ such that $N_1 \unlhd G$.
Then surely $(N_1 \times \pi_2(N)) \cap N \leq N$ and $(N_1 \times \pi_2(N)) \cap N \unlhd G \times H$ as an intersection of normal subgroups.
But since $1 \neq N_1 \leq \pi_1(N)$, there exists some $1 \neq g \in N_1$ and $h \in \pi_2(N)$ such that $(g,h) \in N$.
Therefore $(N_1 \times \pi_2(N)) \cap N \neq 1$ in contradiction with the minimality of $N$.
So $\pi_1(N)$ is a minimal normal subgroup of $G$. Similarly, $\pi_2(N)$ is a minimal normal subgroup of $H$.
This completes the proof.
\end{proof}

\begin{definition} \label{definition_g_function}
If G is abelian p-group and $t \geq 0$ is an integer then g(G, t) := the number of factors of order $\geq p^t$ in the primary decomposition of $G$.
\end{definition}

\begin{lemma} \label{lemma_g_function}
Let G be a finite abelian p-group and $t \geq 0$ an integer. Then $g(G, t) = log_p([G[p^t]:G[p^{t-1}]])$.
\end{lemma}
\begin{proof}
Denote $h(G, t) := log_p([G[p^t]:G[p^{t-1}]])$. Fix some integer $t \geq 0$.
On one hand, for a finite cyclic p-group K, we easily have $g(K,t) = h(K, t)$.
On the other hand, if K and H are finite abelian p-groups we have $g(K \times H, t) = g(K, t) + g(H, t)$ and $h(K \times H, t) = h(K, t) + h(H, t)$.
Thus the equality between g and h is proved by induction.
\end{proof}

\begin{lemma} \label{lemma_subgroups_of_abelian_group}
If $H \leq G = \prod_{i=1}^n \mathbb{Z}_{p^{d_i}}$ for some prime p and integers $d_1,\dotsc,d_n$ then there exist integers $c_1 \leq d_1,\dotsc,c_n \leq d_n$ such that $H \cong \prod_{i=1}^n \mathbb{Z}_{p^{c_i}}$
\end{lemma}
\begin{proof}
An equivalent formulation of the proposition is: for any $t \geq 0$, $g(H, t) \leq g(G, t)$. By Lemma \ref{lemma_g_function}, this is equivalent to $[H[p^t]:H[p^{t-1}]] \leq [G[p^t]:G[p^{t-1}]]$. To prove this we note that:
$H[p^t]/H[p^{t-1}] = H[p^t]/(H[p^t] \cap G[p^{t-1}]) \cong (H[p^t]G[p^{t-1}])/G[p^{t-1}] \leq G[p^t]/G[p^{t-1}]$, where the isomorphism follows from the second isomorphism theorem. This completes the proof.
\end{proof}

\subsection{Linear Algebra}
Consider an $n \times n$ matrix $A$, a list of row indices ${\bf r}=(r_1,\dotsc,r_k)$
and a list of column indices ${\bf c}=(c_1,\dotsc,c_k)$.
We define two submatrices of $A$: a $k \times k$ submatrix $S(A; {\bf r}, {\bf c})$ and
an $(n-k) \times (n-k)$ submatrix $S'(A; {\bf r}, {\bf c})$.
The submatrix $S(A; {\bf r}, {\bf c})$ is obtained by keeping the entries of the intersection of
any row belonging to the list ${\bf r}$ and any column belonging to the list ${\bf c}$.
The submatrix $S'(A; {\bf r}, {\bf c})$ is obtained by keeping the entries of the intersection of
any row \emph{not} belonging to the list ${\bf r}$ and any column \emph{not} belonging to the list ${\bf c}$.
To simplify the formula given in the next Lemma, we let the index of the first row and the first column be 0.

\begin{lemma}[Laplace's Determinant Expansion Theorem] \label{lemma_laplace_expansion} 
Let $A$ be an $n \times n$ matrix. Let ${\bf c} = (c_1, \dotsc, c_k)$ be a list of k column indices, where $1 \leq k < n$ and $0 \leq c_1 < c_2 < \dotsm < c_k < n$.
Then, the determinant of $A$ is given by
$det(A) = (-1)^{|{\bf c}|} \sum_{{\bf r}} (-1)^{|{\bf r}|} det S(A; {\bf r}, {\bf c}) det S'(A; {\bf r}, {\bf c})$ \\
where $|{\bf c}| = c_1 + \dotsc + c_k$, $|{\bf r}| = r_1 + \dotsc + r_k$ and the summation is over all $k$-tuples ${\bf r} = (r_1,\dotsc,r_k)$
for which $0 \leq r_1 < \dotsm < r_k < n$.
\end{lemma}
\begin{proof}
See \cite{laplace_expansion}.
\end{proof}

The proof of the following lemma is due to Robert Israel \cite{israel11}.

\begin{lemma} \label{lemma_invertible_matrix_faithful_decomposition}
Let $M$ be an $n \times n$ invertible matrix and let $1 \leq m \leq n-1$ be an integer.
Then it is possible to permuate the rows of $M$ to obtain a matrix $M'$ of the form
$
M'=\left[
\begin{array}{cc}
A & B \\
C & D
\end{array}\right]
$
where $A$ is an $m \times m$ matrix and $D$ is an $(n-m) \times (n-m)$ matrix and both $A$ and $D$ are invertible.
\end{lemma}
\begin{proof}
Assume, for the sake of contradiction, that no permutation of the rows of $M$ brings it to the desired form.
Then, by plugging ${\bf c} = (0,1,2,\dotsc,m-1)$ into Laplace's Expansion Theorem (Lemma \ref{lemma_laplace_expansion}),
we conclude that $det(M)$ is a sum of terms of the form $\pm det(A)det(D)$ such that in each term either
$det(A)$ is zero or $det(D)$ is zero.
Therefore $det(M)=0$ in contradiction the fact that $M$ is invertible. 
\end{proof}

\begin{lemma} \label{lemma_vector_spaces_in_faithful_mode}
Let $V$ be a vector space of finite dimension and let $V_1,\dotsc,V_n$ be subspaces of $V$ such that:
\begin{enumerate}
\item $\cap_{i=1}^n V_i = \{0\}$.
\item For any $1 \leq i \leq n$ we have $\cap_{j \neq i} V_j \neq \{0\}$.
\item For any $1 \leq i \leq n$ we have $dim(V_i) = dim(V)-1$.
\end{enumerate}
Then, $n=dim(V)$ and there exists a basis $\{v_1,\dotsc,v_n\}$ of $V$ such that for any $1 \leq i \leq n$,
$V_i = span\{v_j \mid 1 \leq j \leq n \wedge j \neq i\}$
\end{lemma}
\begin{proof}
First, for any $1 \leq i \leq n$, we have $1 \leq dim(\cap_{j \neq i}V_j) =
\underbrace{dim(V_i + \cap_{j \neq i}V_j)}_{\leq dim(V)} -
\underbrace{dim(V_i)}_{=dim(V)-1}+
\underbrace{dim(\cap_{j=1}^n V_j)}_{=0} \leq 1$,
where the first inequality is due to hypothesis $(2)$.
Thus $dim(\cap_{j \neq i}V_j) = 1$.
Second, for any $1 \leq i \leq n$ we have $(\cap_{j \neq i}V_j) \cap span(\cup_{j \neq i}\cap_{k \neq j}V_k)
\subset (\cap_{j \neq i}V_j) \cap V_i = \cap_{j=1}^n V_j = \{0\}$.
Combining these 2 facts we conclude that there exists a linearly independent set $\{v_1,\dotsc,v_n\} \subset V$
such that for any $1 \leq i \leq n$ we have $\cap_{j \neq i}V_j = span\{v_i\}$.
Now, for any $1 \leq i \leq n$ we have $span\{v_j \mid j \neq i\} = \sum_{j \neq i}\cap_{k \neq j}V_k
\subset V_i$.
Thus, for any $1 \leq i \leq n$ we have $v_i \not\in V_i$, because otherwise we would have
$v_i \in \cap_{j=1}^nV_j = \{0\}$, a contradiction.
To summarize, we have found a linearly indepdendent set $\{v_1,\dotsc,v_n\} \subset V$ such that for any $1 \leq i, j \leq n$ it holds that $v_j \in V_i \Leftrightarrow i \neq j$.
Assume, for the sake of contradiction, that $\{v_1,\dotsc,v_n\}$ does not span $V$ and let
$w \in V$ be such that $w \not\in span\{v_1,\dotsc,v_n\}$.
Thus, for any $1 \leq i \leq n$ we have $w \in V_i$, because otherwise we would have
$dim(V)-1 = dim(V_i) = dim(V_i \oplus span\{v_i, w\}) - dim(span\{v_i, w\}) \leq div(V) - 2$, a contradiction.
So $\{v_1,\dotsc,v_n\}$ is a basis for $V$ and therefore $n=dim(V)$ and for each $1 \leq i \leq n$
we have $V_i = span\{v_j \mid j \neq i\}$ as desired.
\end{proof}

\section{Acknowledgements}
I would like to express my gratitude to my advisor Professor Alexander Lubotzky for directing me and for reading and correcting my work.
I would also like to thank Jack Schmidt of the University of Kentucky for answering some of my questions through
the internet.

\end{document}